\documentclass[12pt]{amsart}

\usepackage{hyperref}

\usepackage{mathrsfs, amssymb, upgreek}
\usepackage[matrix,arrow,curve]{xy}
\usepackage{longtable}
\makeatletter
\@addtoreset{equation}{section}
\makeatother

\renewcommand{\rho}{\uprho}
\renewcommand{\P}{\mathbb{P}}

\newcommand{\CC}{\mathbb{C}}
\newcommand{\QQ}{\mathbb{Q}}
\newcommand{\Q}{\mathbb{Q}}

\newcommand{\ZZ}{\mathbb{Z}}
\newcommand{\PP}{\mathbb{P}}

\newcommand{\mumu}{{\boldsymbol{\mu}}}

\def \Bir {\mathrm{Bir}}
\def \GL {\mathrm{GL}}
\def \SL {\mathrm{SL}}
\def \PGL {\mathrm{PGL}}
\def \GL {\mathrm{GL}}
\def \Sym {\mathrm{Sym}}
\def \g {\mathrm{g}}
\def \dd {\mathrm{d}}

\newcommand{\xref}[1]{{\rm \ref{#1}}}

\newcommand{\Aut}{\operatorname{Aut}}

\newcommand{\rr}{\operatorname{r}}
\newcommand{\Pic}{\operatorname{Pic}}

\newcommand{\Bs}{\operatorname{Bs}}

\newcommand{\rk}{\operatorname{rk}}
\newcommand{\Cl}{\operatorname{Cl}}
\newcommand{\Cr}{\operatorname{Cr}}

\newtheorem{theorem}{Theorem}[section]
\newtheorem{lemma}[theorem]{Lemma}
\newtheorem{corollary}[theorem]{Corollary}

\newtheorem{question}[theorem]{Question}
\newtheorem{proposition}[theorem]{Proposition}

\theoremstyle{definition}
\newtheorem{example}[theorem]{Example}

\theoremstyle{remark}
\newtheorem{remark}[theorem]{Remark}

\renewcommand\labelenumi{(\roman{enumi})}
\renewcommand\theenumi{(\roman{enumi})}
\newcommand{\comment}[1]{}

\def\le {\leqslant}
\def\ge {\geqslant}

\title{$p$-subgroups in the space Cremona group}

\author{Yuri Prokhorov}
\author{Constantin Shramov}

\thanks{This work was supported by the Russian Science Foundation under grant 14-50-00005.}

\address{
Steklov Mathematical Institute of Russian Academy of Sciences, 8 Gubkina st., Moscow, Russia, 119991
}

\email{prokhoro@mi.ras.ru, costya.shramov@gmail.com}

\begin{document}

\begin{abstract}
We prove that if $X$ is a rationally connected threefold and~$G$
is a $p$-subgroup in the group of birational selfmaps of $X$,
then $G$ is an abelian group generated by at most $3$ elements
provided that $p\ge 17$. We also prove a similar result for $p\ge 11$
under an assumption that $G$ acts on
a (Gorenstein) $G$-Fano threefold, and show that the same holds
for $p\ge 5$ under an assumption that $G$ acts
on a $G$-Mori fiber space.
\end{abstract}

\maketitle
\tableofcontents

\section{Introduction}

Let $\Bbbk$ be an algebraically closed field of characteristic $0$.
The Cremona group $\Cr_r(\Bbbk)$ of rank $r$ is
the group of birational selfmaps
of the projective space $\PP_{\Bbbk}^r$.
Classification of finite subgroups of Cremona groups is a classical problem
which goes back to works of E. Bertini, G. Castelnuovo, S. Kantor, and others.
For Cremona group of rank $2$ there is a classification
of finite subgroups, see \cite{Blanc2009}, \cite{Dolgachev-Iskovskikh};
some results are also known
for non-closed fields \cite{Serre2009}, \cite{Dolgachev-Iskovskikh-2009}, \cite{Yasinsky2016}.

For rank $3$ and higher it is generally understood that no reasonable explicit classification
of finite subgroups is possible. However, there are various
boundedness properties of subgroups
of Cremona groups (or, more generally, subgroups
of birational automorphism groups
of arbitrary rationally connected varieties) that are partially proved and partially expected
to hold, see \cite{Serre2009}, \cite{Serre-2008-2009}, \cite{Popov-Jordan},
\cite{ProkhorovShramov-RC}, \cite{Prokhorov-Shramov-2013}, \cite{Prokhorov-Shramov-JCr3}, \cite{Yasinsky2016a},
and references therein. A particular case that is well studied is the case of
simple non-abelian groups;
one can find a classification of such subgroups of $\Cr_3(\Bbbk)$ in~\cite{Prokhorov2009e}.

Recall that a \emph{$p$-group} is a finite group of order $p^k$, where $p$ is a prime.
Such groups are somewhat opposite to simple non-abelian groups from the point of view of group theory.
However, they have various nice properties, and classifying $p$-subgroups of birational
automorphism groups looks an interesting problem
(see \cite{Beauville2007}, \cite{Prokhorov2011a}, \cite{Prokhorov-2groups}).

J.-P. Serre \cite[\S6]{Serre2009} asked if the following boundedness property
holds for $p$-subgroups of Cremona groups.

\begin{question}
Is it true that for any sufficiently large prime number $p$
(depending on~$r$), every finite
$p$-subgroup of $\Cr_r(\Bbbk)$ is abelian and is generated by at most~$r$ elements?
\end{question}

The positive answer to this question was obtained in~\mbox{\cite[Theorem~5.3]{Serre2009}},
\cite[Theorem~1.10]{ProkhorovShramov-RC}
in dimension at most $3$, and in arbitrary dimension
modulo so-called Borisov--Alexeev--Borisov conjecture; the latter conjecture
was recently proved in~\cite[Theorem~1.1]{Birkar}. Moreover,
the following result is implied by
\cite[Theorem~1.2.4]{Prokhorov-Shramov-JCr3}.

\begin{theorem}
\label{theorem:too-large}
Let $X$ be a rationally connected threefold and
$G$ be a finite $p$-subgroup of $\Bir(X)$.
Suppose that $p>10368$.
Then $G$ is an abelian group generated by at most $3$ elements.
\end{theorem}

If we restrict ourselves to the case of abelian $p$-groups, more precise results are available.

\begin{theorem}[{\cite{Beauville2007}}]
\label{theorem:Beauville}
Let $G$ be an abelian $p$-subgroup of $\Cr_2(\Bbbk)$.
Then~$G$ is generated by at most~$r$ elements,
where
\[
r\le
\begin{cases}
4& \text{if $p=2$,}
\\
3& \text{if $p=3$,}
\\
2& \text{if $p\ge 5$.}
\end{cases}
\]
For any $p$ this bound is attained for some
abelian $p$-sub\-group~\mbox{$G\subset \Cr_2(\Bbbk)$}.
\end{theorem}

\begin{theorem}[{\cite{Prokhorov2011a}}, {\cite{Prokhorov-2groups}}]
\label{theorem:elementary-p-groups}
Let $X$ be a rationally connected threefold and
$G$ be an abelian $p$-subgroup of $\Bir(X)$.
Then $G$ is generated by at most~$r$ elements,
where
\[
r\le
\begin{cases}
6& \text{if $p=2$,}
\\
5& \text{if $p=3$,}
\\
4& \text{if $p=5$, $7$, $11$, or $13$,}
\\
3& \text{if $p\ge 17$.}
\end{cases}
\]
For any $p\ge 17$ and for $p=2$ this bound is attained for some abelian
$p$-sub\-group~\mbox{$G\subset \Cr_3(\Bbbk)$}.
\end{theorem}

The goal of this paper is to prove the following refinement of
Theorem~\ref{theorem:too-large}.

\begin{theorem}\label{theorem:main}
Let $X$ be a rationally connected threefold and
$G\subset\Bir(X)$ be a $p$-subgroup. Suppose that $p\ge 17$.
Then $G$ is an abelian group
generated by at most three elements.
\end{theorem}

Actually, we deduce Theorem~\ref{theorem:main} from the following fact
that looks interesting on its own.
\begin{theorem}[{cf.  \cite[Theorem~4.2]{ProkhorovShramov-RC}}]
\label{theorem:fixed-point}
Let $X$ be a projective rationally connected threefold and
$G\subset\Aut(X)$ be a $p$-subgroup. Suppose that $p\ge 17$.
Then $G$ has a fixed point on $X$.
\end{theorem}

In course of the proof of Theorem~\ref{theorem:main}, we also establish the following
result in dimension~$2$ (this can be deduced from the classification
of finite subgroups of $\Cr_2(\Bbbk)$ obtained by I.~Dolgachev
and V.~Iskovskikh in~\cite{Dolgachev-Iskovskikh}, but our proof is
short and self-contained).

\begin{proposition}\label{proposition:Cr2}
Let $S$ be a rational surface, and $G\subset\Cr_2(\Bbbk)$ be a $p$-subgroup.
Suppose that $p\ge 5$. Then $G$ is an abelian group generated by at most two elements.
\end{proposition}

Similarly to the case of Theorem~\ref{theorem:main}, we derive Proposition~\ref{proposition:Cr2}
from the following fact.

\begin{proposition}\label{proposition:Cr2-fixed-point}
Let $S$ be a projective rational surface, and $G\subset\Aut(S)$ be a $p$-group.
Suppose that $p\ge 5$. Then $G$ has a fixed point on $S$.
\end{proposition}

According to \cite[Theorem~4.2]{ProkhorovShramov-RC} and~\cite[Theorem~1.1]{Birkar},
a generalization of Proposition~\ref{proposition:Cr2-fixed-point} and Theorem~\ref{theorem:fixed-point}
holds in any dimension; however, obtaining an explicit bound here
seems to be out of reach for our techniques.

At the moment we are not aware
of any examples of $p$-groups violating the assertion of Theorem~\ref{theorem:main}
for~\mbox{$p\ge 5$}. This makes us suspect that the bound provided
by Theorem~\ref{theorem:main} may be improved.
Also, the following question looks interesting.

\begin{question}
Let $X$ be a projective rationally connected variety of dimension $r$, and $G\subset\Aut(X)$
be a $p$-subgroup that cannot be generated by less than~$r$ elements.
Suppose that $p$ is sufficiently large.
Is it true that $X$ is rational?
Is it true that $G$ is contained in an $r$-dimensional torus acting on~$X$?
\end{question}

\smallskip
The plan of the paper is as follows.
In~\S\ref{section:examples}
we consider several easy examples of ``large'' $2$- and $3$-groups acting on rational surfaces
and rationally connected threefolds. In~\S\ref{section:linear-algebra}
we collect auxiliary facts about linear representations of $p$-groups.
In~\S\ref{section:extremal}
we discuss lifting properties for fixed points of $p$-groups
along extremal contractions following~\cite[\S3]{ProkhorovShramov-RC}.
In~\S\ref{section:surfaces}
we study fixed points of $p$-groups on rational surfaces, and
prove Propositions~\ref{proposition:Cr2-fixed-point} and~\ref{proposition:Cr2};
as a consequence, we prove Corollary~\ref{corollary:MFS-dim-3}, that is a more precise analog
of Theorem~\ref{theorem:main} in this case.
In~\S\ref{section:anticanonical}
we describe invariant anticanonical divisors for $p$-groups
acting on Fano threefolds, and prove Corollary~\ref{corollary:Gorenstein-Fano-3-folds},
that is a more precise analog of Theorem~\ref{theorem:main} in this case. In~\S\ref{section:Gorenstein}
we study fixed points of $p$-groups on Gorenstein Fano threefolds.
In~\S\ref{section:proof}
we prove Theorems~\ref{theorem:fixed-point} and~\ref{theorem:main}
and make concluding remarks.

Throughout the paper we work over an algebraically closed field
$\Bbbk$ of characteristic~$0$. All varieties are assumed to be projective.
By $\mumu_n$ we denote the cyclic group of order~$n$.
We will always denote by $p$ some prime number.
For a finite $p$-group $G$ we denote the minimal number
of its generators by~\mbox{$\rr(G)$}.

We are grateful to J.-P.~Serre for his interest in our work, and to
A.~Kuznetsov and L. Rybnikov for useful discussions.

\section{Examples}
\label{section:examples}

For an arbitrary integer $n$, the group $\GL_r(\Bbbk)\subset\Cr_r(\Bbbk)$
contains a subgroup isomorphic to $\mumu_n^r$. In this section we give several examples
of more complicated $p$-subgroups of Cremona groups for~\mbox{$p=2$} and~\mbox{$p=3$}.

Recall that an \emph{elementary} abelian $p$-group is a group isomorphic to
$\mumu_p^r$ for some $r$.

\begin{remark}\label{remark-Frattini}
Let $G$ be a finite $p$-group and let
$\Phi(G)$ be its Frattini subgroup, that is, the intersection of all maximal
proper subgroups of $G$.
Then~\mbox{$G/\Phi(G)$} is an elementary abelian $p$-group.
By Burnside's basis theorem~\mbox{\cite[Theorem~12.2.1]{Hall1976}} we have $\rr(G)=\rr(G/\Phi(G))$.
\end{remark}

\begin{example}
Let $\mathfrak Q\cong \mumu_2\times\mumu_2$ be the subgroup
of $\Aut(\P^1)\cong\PGL_2(\Bbbk)$
generated by the matrices
\[
\begin{pmatrix}
1& \phantom{-}0 \\ 0& - 1
\end{pmatrix},
\qquad
\begin{pmatrix}
0&1\\ 1&0
\end{pmatrix}.
\]
Then $\Aut(\P^1\times\P^1\times\P^1)$ contains a non-abelian subgroup
$$
G=(\mathfrak Q\times\mathfrak Q\times\mathfrak Q)\rtimes \mumu_2,
$$
where $\mumu_2$ acts on
$\P^1\times\P^1\times\P^1$ by permutation of the first two factors.
Using Remark~\ref{remark-Frattini}, one can compute that $\rr(G)=5$.
Note that for the
group~\mbox{$G'=\mathfrak Q\times\mathfrak Q\times\mathfrak Q$}
we have $G'\cong\mumu_2^6$, so that $\rr(G')>\rr(G)$.
The groups~$G$ and~$G'$ have no fixed points on~\mbox{$\P^1\times\P^1\times\P^1$}.
\end{example}

\begin{example}\label{example:3-group-on-P3-and-Fermat}
Let $\Gamma'\subset\PGL_4(\Bbbk)$ be the group that is generated by multiplications of the
homogeneous coordinates
by cubic roots of $1$, and~\mbox{$\Gamma\subset\PGL_4(\Bbbk)$} be the group
generated by $\Gamma'$ and a cyclic permutation of the first three homogeneous coordinates.
Then $\Gamma'\cong\mumu_3^3$, while $\Gamma$ is a non-abelian
group with $\rr(\Gamma)=2$.
Let $X\subset\P^3$ be the Fermat cubic surface, i.e., the cubic surface
given by equation
$$
x^3+y^3+z^3+w^3=0.
$$
Then $X$ is invariant with respect to the group $\Gamma$.
The groups~$\Gamma$ and~$\Gamma'$ have no fixed points on~$X$.
A rational variety
$X\times\P^1$ is acted on by a non-abelian $3$-group
$\Gamma\times\mumu_3$ with $\rr(\Gamma\times\mumu_3)=3$,
and also by its subgroup $\Gamma'\times\mumu_3\cong\mumu_3^4$.
The groups~\mbox{$\Gamma\times\mumu_3$} and~\mbox{$\Gamma'\times\mumu_3$} have no fixed points on~\mbox{$X\times\P^1$}.
\end{example}

\begin{example}
Similarly to Example~\ref{example:3-group-on-P3-and-Fermat}, consider
a Fermat cubic threefold~$Y$ in $\P^4$. Then $Y$ is a smooth non-rational
rationally connected threefold, and there is an action of the group
$\Gamma\times\mumu_3$ on~$Y$.
The group~\mbox{$\Gamma\times\mumu_3$} has no fixed points on~$Y$.
\end{example}

\section{Representations}
\label{section:linear-algebra}

In this section we collect auxiliary assertions about linear and projective representations of $p$-groups.
The following facts are easy and well-known.

\begin{lemma}
\label{lemma:basic}
Let $G$ be a $p$-group.
Then the following assertions hold.
\begin{itemize}
\item[(i)]
For any $k$ such that $p^k$ divides $|G|$ the group
$G$ contains a subgroup of order $p^k$.

\item[(ii)] Any non-trivial normal subgroup of $G$ has a non-trivial intersection
with the center of $G$.

\item[(iii)] Let $V$ be a representation
of $G$ defined over~$\QQ$ such that $G$ acts non-trivially in~$V$.
Then~\mbox{$\dim V\ge p-1$}.

\item[(iv)] Suppose that $G$ is non-abelian.
Let $V$ be a $G$-representation such that the center of $G$
acts faithfully on $V$.
Then $V$ contains an irreducible $G$-subrepresentation $W$
such that $\dim W$ is divisible by~$p$.

\item[(v)] Suppose that $G$ is non-abelian, and~\mbox{$p\ge 3$}.
Let $V$ be a faithful representation of a $G$ defined over~$\QQ$.
Then~\mbox{$\dim V\ge 2p$}.
\end{itemize}
\end{lemma}

\begin{remark}\label{remark:detskiy-sad}
Let $G\subset\GL_n(\Bbbk)$ be a finite abelian group.
Then its elements are simultaneously
diagonalizable, so that $G$ is generated by at most $n$ elements.
\end{remark}

\begin{lemma}\label{lemma:GL}
Let $G\subset\GL_n(\Bbbk)$ be a $p$-group. Suppose that~\mbox{$p>n$}.
Then $G$ is abelian and $\rr(G)\le n$.
\end{lemma}
\begin{proof}
Suppose that the group $G$ is not abelian.
By Lemma~\ref{lemma:basic}(iv)
the natural $n$-dimensional $G$-representation $V$
has a subrepresentation whose dimension is divisible
by $p>n$,
which is a contradiction.
Therefore, $G$ is an abelian group, so that $\rr(G)\le n$ by Remark~\ref{remark:detskiy-sad}.
\end{proof}

Denote by $T_P(U)$ the Zariski tangent space
to a variety $U$ at a point~\mbox{$P\in U$}.

\begin{remark}
\label{remark:Aut-P}
Let $U$ be an irreducible variety acted on by a finite group~$G$.
Suppose that $G$ has a fixed point $P\in U$.
Then there is an embedding
$$G\hookrightarrow \GL\big(T_P(U)\big),$$
see e.\,g.~\cite[Lemma~2.4]{Bialynicki-Birula1973}, \cite[Lemma~4]{Popov-Jordan}.
\end{remark}

\begin{lemma}\label{lemma:PGL}
Let $G\subset\Aut(\P^{n-1})\cong\PGL_n(\Bbbk)$ be a $p$-group.
Suppose that~\mbox{$p>n$}. Then $G$ has a fixed point on $\P^{n-1}$.
\end{lemma}
\begin{proof}
Consider the natural projection $\theta\colon\SL_n(\Bbbk)\to\PGL_n(\Bbbk)$,
and let $\tilde{G}$ be the preimage of $G$ with respect to~$\theta$.
Let $\bar{G}$ be a Sylow $p$-subgroup of~$\tilde{G}$.
We have $\theta(\bar{G})=G$.
On the other hand, $\bar{G}$ is an abelian
group by Lemma~\ref{lemma:GL}. Thus $\bar{G}$ has a one-dimensional
subrepresentation in its natural $n$-dimensional representation,
which means that $G$ has a fixed point in $\P^{n-1}$.
By Remark~\ref{remark:Aut-P} this implies that
$G$ is a subgroup of $\GL_{n-1}(\Bbbk)$, so that
the remaining assertions follow from Lemma~\ref{lemma:GL}.
\end{proof}

\begin{corollary}
\label{corollary:rational-curve}
Let $G$ be a $p$-group acting faithfully on a rational curve $C$.
Suppose that $p\ge 3$. Then $G$ has a fixed point on $C$.
\end{corollary}
\begin{proof}
The action of $G$ lifts to the normalization $\tilde{C}$ of $C$,
and $G$ has a fixed point on $\tilde{C}$ by Lemma~\ref{lemma:PGL}.
\end{proof}

\begin{corollary}
\label{corollary:conic}
Let $G$ be a $p$-group acting faithfully on a (possibly reducible) conic $C$.
Suppose that $p\ge 3$. Then $G$ has a fixed point on $C$.
\end{corollary}
\begin{proof}
The conic $C$ is either isomorphic to $\P^1$,
or is a union of two irreducible components meeting in a single point.
In the former case the group $G$ has a fixed point on $C$
by Lemma~\ref{lemma:PGL}, and in the latter case the intersection
point of the irreducible components is $G$-invariant.
\end{proof}

\begin{corollary}\label{corollary:P1xP1}
Let $G\subset\Aut(\P^1\times\P^1)$ be a $p$-group. Suppose that $p\ge 3$.
Then $G$ has a fixed point on $\P^1\times\P^1$.
\end{corollary}
\begin{proof}
One has
$$
\Aut\big(\P^1\times\P^1\big)\cong\big(\PGL_2(\Bbbk)\times\PGL_2(\Bbbk)\big)\rtimes \mumu_2.
$$
Hence $G\subset\PGL_2(\Bbbk)\times\PGL_2(\Bbbk)$, so that
the assertion follows from Lemma~\ref{lemma:PGL}.
\end{proof}

\begin{corollary}\label{corollary:quadric-surface}
Let $G\subset\PGL_4(\Bbbk)$ be a $p$-group. Suppose that there is a $G$-invariant surface
$S$ of degree at most $2$ in $\P^4$. Suppose also that $p\ge 3$.
Then $G$ has a fixed point on $S$.
\end{corollary}
\begin{proof}
If $\deg(S)=1$, then $G$ has a fixed point on $S$ by Lemma~\ref{lemma:PGL}.
Thus we may assume that~\mbox{$\deg(S)=2$}.
If~$S$ is reducible, then both its irreducible components are $G$-invariant,
so that there is again a fixed point on~$S$. If~$S$ is a cone with an irreducible
base, then its unique singular point is fixed by~$G$. Finally, if~$S$ is a non-singular quadric,
then $G$ has a fixed point on $S$ by Corollary~\ref{corollary:P1xP1}.
\end{proof}

Recall that for a given group $G$ with a representation in a
vector space~$V$ a \emph{homogeneous semi-invariant} of $G$ of degree $n$
is a non-zero homogeneous
polynomial $f\in\Bbbk[V]$ of degree $n$ the vector space $\Bbbk\cdot f$ is $G$-invariant.
In other words, $f$ is an eigen-vector of $G$ in the vector space~\mbox{$\Sym^n V^{\vee}$}.

\begin{lemma}\label{lemma:Heisenberg}
Let $V$ be a $p$-dimensional vector space, and $G\subset\GL(V)$ be a non-abelian $p$-group.
Then $G$ has no semi-invariants of degree less than~$p$.
\end{lemma}
\begin{proof}
By Lemma~\ref{lemma:basic}(iv) the representation $V$ is irreducible.
Therefore, by Schur's lemma the center $Z$ of $G$ acts on $V$ by scalar
matrices whose entries are roots of $1$ of degree $p^k$ for some $k$.
In particular, $Z$ acts faithfully on~\mbox{$\Sym^r V^\vee$} for every
$r$ coprime to $p$.

Now suppose that $G$ has a homogeneous semi-invariant of positive degree~\mbox{$d<p$}.
This means that $\Sym^d V^\vee$ contains a one-dimensional
$G$-subrepresentation $W$. Since $Z$ acts (faithfully) by scalar matrices
on $\Sym^d V^\vee$, it acts faithfully on $W$ as well. The latter
contradicts Lemma~\ref{lemma:basic}(iv).
\end{proof}

\begin{lemma}\label{lemma:L-cap-R}
Let $V$ be a vector space, and $G\subset\GL(V)$ be an abelian $p$-group.
Suppose that $G$ preserves a hypersurface $R$ in $\P(V)$ of degree
less than $p$. Then $G$ has a fixed point on $R$.
\end{lemma}
\begin{proof}
Since $G$ is abelian, $V$ splits into a sum of one-dimensional subrepresentations.
Choose some two-dimensional $G$-invariant subspace $W\subset V$, and let $L=\P(W)$
be the corresponding line in $\P(V)$. If $L$ is not contained in $R$, then
the intersection $L\cap R$ is a non-empty $G$-invariant set that consists of less than $p$ points,
so that all these points must be $G$-invariant. Thus we may assume that $L\subset R$.
It remains to notice that $L$ contains a point fixed by $G$ since $W$ contains a one-dimensional subrepresentation.
\end{proof}

\begin{lemma}\label{lemma:proj-Heisenberg}
Let $V$ be a vector space of dimension at most $p$, and let~\mbox{$G\subset\PGL(V)$} be a $p$-group.
Suppose that $G$ preserves a hypersurface~$R$ in $\P(V)$ of degree
less than $p$. Then $G$ has a fixed point on $R$.
\end{lemma}
\begin{proof}
Let $\theta\colon\SL(V)\to\PGL(V)$ be the natural homomorphism.
Put~\mbox{$\tilde{G}=\theta^{-1}(G)$}, and let $\bar{G}$ be a Sylow $p$-subgroup
of $\tilde{G}$. Then $\bar{G}$ is a $p$-group.
If $\dim V<p$, then $\bar{G}$ is abelian by Lemma~\ref{lemma:GL}.
If $\dim V=p$, then~$\bar{G}$ is abelian by Lemma~\ref{lemma:Heisenberg}.
Thus $G=\theta(\bar{G})$ is abelian as well, and there is a $\bar{G}$-fixed (that is, $G$-fixed)
point on $R$ by Lemma~\ref{lemma:L-cap-R}.
\end{proof}

\begin{lemma}\label{lemma:elliptic}
Let $V$ be a vector space of dimension $n\ge 3$, and
let~\mbox{$C\subset\P(V)$} be an irreducible curve of degree $n+1$ and geometric genus~$1$. Let
$G\subset \PGL(V)$ be a $p$-group such that $C$ is $G$-invariant.
Then $p\le n+1$.
\end{lemma}
\begin{proof}
Suppose that $p>n+1$. Replacing $V$ by its $G$-invariant linear subspace,
we may assume that $C$ is not contained in a hyperplane in $\P(V)$, so that the action of
$G$ on $C$ is faithful.

Let $\nu\colon \tilde{C}\to C$ be the normalization of $C$,
and put~\mbox{$\mathcal{L}=\nu^*|\mathcal{O}_{\P(V)}(1)\vert_C|$}.
Then $\tilde{C}$ is an elliptic curve with an action of $G$, and $\mathcal{L}$ is an $n$-dimensional $G$-invariant linear system on $\tilde{C}$. By Lemma~\ref{lemma:PGL} there is a
$G$-invariant divisor~$L$ in~$\mathcal{L}$.

Since $p>3$, we see that $G$ acts on $\tilde C$ by translations.
Therefore, the degree of $L$ is divisible by $p$. This is impossible by the Riemann--Roch theorem.
\end{proof}

\begin{corollary}\label{corollary:curve-deg-le-3}
Let $C\subset\P^2$ be a curve of degree at most $3$, and let~\mbox{$G\subset\PGL_3(\Bbbk)$} be a $p$-group such that $C$ is
$G$-invariant. Suppose that~\mbox{$p\ge 5$}. Then $G$ has
a fixed point on $C$.
\end{corollary}
\begin{proof}
Replacing $C$ by one of its irreducible components if necessary, we may assume that $C$ is irreducible.
If $C$ is a line, then $G$ has a fixed point on $C$ by Lemma~\ref{lemma:PGL}.
If $C$ is a conic, then $G$ has a fixed point on $C$ by Corollary~\ref{corollary:conic}.
Finally, if $C$ is an (irreducible) cubic, then
it is singular by Lemma~\ref{lemma:elliptic}; so $C$ has a unique singular
point, which must be fixed by~$G$.
\end{proof}

\begin{corollary}\label{corollary:curve-deg-le-4}
Let $C\subset\P^3$ be a curve of degree at most $4$.
Suppose that for a general point $P$ of (every irreducible component of) $C$
there is a neighbourhood $U_P$ such that $C\cap U_P$ is cut out by quadrics.
Let $G\subset\PGL_4(\Bbbk)$ be a $p$-group such that $C$ is
$G$-invariant. Suppose that $p\ge 5$. Then $G$ has
a fixed point on $C$.
\end{corollary}
\begin{proof}
Replacing $C$ by one of its irreducible components if necessary, we may assume that $C$ is irreducible.
If $C$ is a line, then $G$ has a fixed point on $C$ by Lemma~\ref{lemma:PGL}.
If $C$ is a conic, then $G$ has a fixed point on $C$ by Corollary~\ref{corollary:conic}.
Therefore, we may assume that $C$ is not contained in a hyperplane in $\P^3$.
If $C$ is a twisted cubic, then $G$ has a fixed point on $C$ by~Lemma~\ref{lemma:PGL}.
If $C$ is a rational quartic, then $G$ has a fixed point on $C$
by Corollary~\ref{corollary:rational-curve}. Finally, if
$C$ is a normal elliptic curve of degree $4$, then $G$ has a fixed point on $C$ by Lemma~\ref{lemma:elliptic}.
\end{proof}

\begin{lemma}\label{lemma:Heisenberg-plus-one}
Let $V$ be a $(p+r)$-dimensional vector space, where~\mbox{$1\le r\le p-1$}.
Let $G\subset\GL(V)$ be a non-abelian $p$-group.
Suppose that $G$ has a homogeneous semi-invariant $f$ of degree $1\le d\le p-1$,
and let $R\subset\P(V)$ be the subscheme
defined by equation $f=0$.
The following assertions hold:
\begin{itemize}
\item[(i)] if $r=1$, then $f$ is a $d$-th power of a linear form;

\item[(ii)] if $r\ge 2$ and $d\ge 2$, then the singular locus of $R$ has dimension
at least~\mbox{$p-1$}.
\end{itemize}
\end{lemma}
\begin{proof}
By Lemma~\ref{lemma:basic}(iv) the representation $V$
splits into a sum of an irreducible $p$-dimensional
representation $U$ and $r$ one-dimensional representations $T_1,\ldots,T_r$.

The semi-invariant $f$ gives a one-dimensional
$G$-subrepresentation~$W$ in~\mbox{$\Sym^d V^\vee$}.
One has a splitting
\begin{equation}\label{eq:splitting}
\Sym^d V^\vee=\bigoplus W_{k_0,\ldots,k_r}, \quad k_0+\ldots+k_r=d,\ k_i\ge 0,
\end{equation}
where
\begin{equation*}
W_{k_0,\ldots,k_r}=\Sym^{k_0} U^\vee\otimes\Sym^{k_1} T_1^\vee\otimes\ldots\otimes \Sym^{k_r} T_r^\vee.
\end{equation*}
This splitting agrees with the action of $G$, so that there
are $G$-equivariant projectors on each of
the summands in~\eqref{eq:splitting}.

Assume that a projection of $W$ to at least one subspace $W_{k_0,\ldots,k_r}$
with $k_0\ge 1$ is non-trivial.
This gives a one-dimensional subrepresentation in~$W_{k_0,\ldots,k_r}$, and thus in $\Sym^{k_0}U^\vee$.
The latter is impossible by Lemma~\ref{lemma:Heisenberg}.

Therefore, we see that $W$ is contained in the subspace
$$
\bigoplus_{k_1+\ldots+k_r=d} W_{0,k_1,\ldots,k_r}\subset\Sym^d V^\vee.
$$
If $r=1$, this means that $f$ is a $d$-th power of a linear form, which
is assertion~(i).
If $r>1$, this means that $R$ is a cone over a subscheme
in the linear subspace $\P(T_1\oplus\ldots\oplus T_r)\subset\P(V)$ with the vertex $\P(U)\subset\P(V)$;
under an additional assumption that $d\ge 2$ we see that
$R$ is singular
along~\mbox{$\P(U)$},
which is assertion~(ii).
\end{proof}

The following result will be used in the proofs of
Lemmas~\ref{lemma:rho-1-iota-large} and~\ref{lemma:Fano-3-fold-small-g}.

\begin{lemma}\label{lemma:complete-int-of-two-quadrics-or-2-3}
Let $V$ be a vector space of dimension $n\ge 3$,
and~\mbox{$G\subset\PGL(V)$} be a $p$-group.
Suppose that $G$ preserves a normal variety~$X$
such that $X$ is either a complete intersection of two quadrics
in~\mbox{$\P(V)$}, or a complete intersection
of a quadric and a cubic in $\P(V)$. Suppose that~\mbox{$p\ge n-1$}
and $p\ge 5$.
Then $G$ has a fixed point on $X$.
\end{lemma}
\begin{proof}
If $X$ is a complete intersection of a quadric and a cubic,
then there is a unique quadric $Q$ passing through $X$, so that $Q$
is $G$-invariant. If $X$ is a complete intersection of two quadrics, then
$G$ acts on the pencil of quadrics passing through $X$;
thus Lemma~\ref{lemma:PGL} implies that there is a $G$-invariant
quadric $Q$ passing through $X$ in this case as well. Since $X$ is normal,
the codimension of the singular locus of $Q$ is at least $2$, so that
the quadric $Q$ is reduced.

Let $\theta\colon\SL(V)\to\PGL(V)$ be the natural homomorphism.
Put~\mbox{$\tilde{G}=\theta^{-1}(G)$}, and let $\bar{G}$ be a Sylow $p$-subgroup
of $\tilde{G}$. One has~\mbox{$\theta(\bar{G})=G$}.

If $p>n$, then the group $\bar{G}$ is abelian
by Lemma~\ref{lemma:GL}. If $p=n$, then $\bar{G}$ is abelian by Lemma~\ref{lemma:Heisenberg}. If $p=n-1$, then $\bar{G}$ is abelian by Lemma~\ref{lemma:Heisenberg-plus-one}(i).

Let $W$ be a three-dimensional $\bar{G}$-invariant subspace of $V$,
and denote by~\mbox{$\Pi\subset\P(V)$} its projectivization.
If $\Pi$ is contained in $X$, then $X$ contains a point fixed by $G$
since $W$ contains a one-dimensional $\bar{G}$-subrepresentation.
Thus we may assume that $\Pi$ is not contained in $X$.

If $\Pi\cap X$ is one-dimensional, let $C$ be the union of its one-dimensional
irreducible components. Then $\deg(C)\le 3$, so that
$G$ has a fixed point on~$C$ by Corollary~\ref{corollary:curve-deg-le-3}.

Thus we may assume that the intersection $\Pi\cap X$ is finite. This means that it consists
of either $4$ or $6$ points (counted with multiplicities).
One of them must be fixed by the group $G$
because $p\ge 5$.
\end{proof}

The following result will be used in the proof of
Lemma~\ref{lemma:Fano-3-fold-small-g}.

\begin{lemma}\label{lemma:three-quadrics}
Let $V$ be a vector space of dimension $n\ge 4$,
and~\mbox{$G\subset\PGL(V)$} be a $p$-group.
Suppose that $G$ preserves a variety $X$
that is a complete intersection of three quadrics in $\P(V)$.
Suppose that $X$ is normal.
Suppose also that~\mbox{$p\ge n-2$} and $p\ge 5$.
Then $G$ has a fixed point on $X$.
\end{lemma}
\begin{proof}
The group $G$ acts on the projective plane parameterizing quadrics passing through
$X$; thus Lemma~\ref{lemma:PGL} implies that there is a $G$-invariant
quadric $Q$ passing through $X$. Since $X$ is normal, we see that
the codimension of the singular locus of $Q$ is at least~$2$, i.e.
the singular locus of $Q$ has dimension at most~\mbox{$n-4$}.

Let $\theta\colon\SL(V)\to\PGL(V)$ be the natural homomorphism.
Put~\mbox{$\tilde{G}=\theta^{-1}(G)$}, and let $\bar{G}$ be a Sylow $p$-subgroup
of $\tilde{G}$. One has~\mbox{$\theta(\bar{G})=G$}.
The group $\bar{G}$ is abelian by Lemma~\ref{lemma:GL} if
$p>n$, by Lemma~\ref{lemma:Heisenberg} if $p=n$,
by Lemma~\ref{lemma:Heisenberg-plus-one}(i) if $p=n-1$ and by Lemma~\ref{lemma:Heisenberg-plus-one}(ii)
if $p=n-2$.
Let $W$ be a four-dimensional $\bar{G}$-invariant subspace of $V$,
and let $\Pi\subset\P(V)$ be its projectivization.
If $\Pi$ is contained in $X$, then $X$ contains a point fixed by $G$
since $W$ contains a one-dimensional $\bar{G}$-subrepresentation.
Thus we may assume that $\Pi$ is not contained in $X$.

If $\Pi\cap X$ is two-dimensional, let $S$ be the union of its two-dimensional
irreducible components. Then $\deg(S)\le 2$, so that $G$
has a fixed point on~$S$
by Corollary~\ref{corollary:quadric-surface}.

If $\Pi\cap X$ is one-dimensional, let $C$ be the union of its one-dimensional
irreducible components. Then $\deg(C)\le 4$, and for a general point $P$ of (every irreducible component of) $C$ there is a neighbourhood $U_P$ such that $C\cap U_P$ is cut out by quadrics.
so that $G$ has a fixed point on $C$ by Corollary~\ref{corollary:curve-deg-le-4}.

Thus we may assume that the intersection $\Pi\cap X$ is finite. This means that it consists
of $8$ points (counted with multiplicities).
One of them must be fixed by the group $G$
because $p\ge 5$.
\end{proof}

The following result will be used in the proof of Lemma~\ref{lemma:g-ge-9}.

\begin{lemma}\label{lemma:g-3-deg-7}
Let $\mumu_5\subset\PGL_4(\Bbbk)$ be a subgroup
such that there is a smooth $\mumu_5$-invariant
curve $C\subset\P^3$ of genus $3$. Then $\deg(C)\neq 7$.
\end{lemma}
\begin{proof}
Suppose that $\deg(C)=7$. Note that $C$ is not contained in a plane.
Hence the action of $\mumu_5$ on $C$ is faithful.
It follows from the Hurwitz formula that $\mumu_5$ has a unique
fixed point on $C$. Let us denote it by $P$.

It is easy to see that there is an embedding $\mumu_5\subset\GL_4(\Bbbk)$
that induces our action of $\mumu_5$ on $\P^3$;
for instance, one can take any preimage~$\tilde{g}$ in~\mbox{$\SL_4(\Bbbk)$}
of a generator $g$ of $\mumu_5\subset\PGL_4(\Bbbk)$, and consider the group
generated by~$\tilde{g}^4$.
The corresponding
four-dimensional vector space
splits as a sum of four one-dimensional $\mumu_5$-representations.
Let $H_1,\ldots,H_4$ be $\mumu_5$-invariant planes in $\P^3$
such that $H_1\cap\ldots\cap H_4=\varnothing$. Then
the intersection of~$H_i$ with $C$ consists of $7$ points
(counted with multiplicities), so that~\mbox{$H_i\cap C$} contains
at least one $\mumu_5$-fixed point, say,~$P_i$. At least one of the
points~\mbox{$P_1,\ldots,P_4$} is different from~$P$, which gives a contradiction.
\end{proof}

\section{Extremal contractions}
\label{section:extremal}

In this section we adapt the results of \cite[\S3]{ProkhorovShramov-RC}
for $p$-groups. We will use the following notation. Let $L(n)$ be a minimal integer
such that for any rationally connected variety $X$ of dimension $n$
and any $p$-group~\mbox{$G\subset\Aut(X)$} with~\mbox{$p>L(n)$}
the group $G$ has a fixed point on~$X$.
This definition makes sense by \cite[Theorem~4.2]{ProkhorovShramov-RC} and~\cite[Theorem~1.1]{Birkar}.
Note that~\mbox{$L(1)=2$} by Corollary~\ref{corollary:rational-curve}.

\begin{lemma}
\label{lemma:G-MMP}
Let $X$ be a variety of dimension $n$, and $G\subset\Aut(X)$ be a $p$-group.
Suppose that~$X$ has terminal $G\Q$-factorial singularities.
Let~\mbox{$f\colon X\dasharrow Y$} be a birational map that is a result of a
$G$-Minimal Model Program ran on~$X$.
Suppose that $G$ has a fixed point on~$Y$, and~\mbox{$p>L(n-1)$}.
Then $G$ has a fixed point on~$X$.
\end{lemma}
\begin{proof}
The rational map $f$ is a composition of $G$-contractions
(see~\mbox{\cite[\S2]{ProkhorovShramov-RC}} for a precise definition)
and $G$-flips, so it is enough to prove the assertion for a
$G$-contraction and for a $G$-flip.
If $f\colon X\to Y$ is a $G$-contraction, then
there is a $G$-invariant rationally connected
subvariety~\mbox{$Z\subsetneq X$}
by~\cite[Corollary~3.7]{ProkhorovShramov-RC}.
If $f\colon X\to Y$ is a $G$-flip, then
there is a $G$-invariant rationally connected subvariety
$Z\subsetneq X$ by~\cite[Corollary~3.8]{ProkhorovShramov-RC}.
In any case, one has~\mbox{$\dim Z\le n-1$}, so that $G$ has a fixed point on $Z$.
\end{proof}

Recall that a $G$-equivariant morphism $\phi\colon X\to S$ of normal
varieties acted on by a finite group
$G$ is a \emph{$G$-Mori fiber space},
if $X$ has terminal $G\Q$-factorial singularities, one has~\mbox{$\dim(S)<\dim(X)$},
the fibers of $\phi$ are connected, the anticanonical divisor~$-K_X$
is $\phi$-ample, and the relative $G$-invariant Picard number
$\rho^G(X/S)$ equals~$1$.
If the dimension of~$X$ equals~$3$, there are three cases:
\begin{itemize}
\item
$S$ is a point, $-K_X$ is ample; in this case $X$ is said to be a $G\Q$-Fano threefold,
and $X$ is a $G$-Fano threefold provided that the singularities of $X$ are Gorenstein;
\item
$S$ is a curve, a general fiber of $\phi$ is a del Pezzo surface; in this case $X$ is said to be a $G\Q$-del Pezzo fibration;
\item
$S$ is a surface, the general fiber of $\phi$ is a rational curve; in this case $X$ is said to be a $G\Q$-conic bundle.
\end{itemize}
Similarly to Lemma~\ref{lemma:G-MMP}, we prove the following result.

\begin{corollary}\label{corollary:MFS}
Let $X$ be a rationally connected variety, and $G\subset\Aut(X)$ be a $p$-group.
Suppose that $\phi\colon X\to S$ is a $G$-Mori fiber space with $\dim S>0$.
Suppose that $p>L(n-1)$.
Then $G$ has a fixed point on $X$.
\end{corollary}
\begin{proof}
Since $X$ is rationally connected, $S$ is rationally connected as well.
Thus the group $G$ has a fixed point on $S$.
Hence there is a $G$-invariant rationally connected subvariety
$Z\subsetneq X$ by~\cite[Corollary~3.7]{ProkhorovShramov-RC}.
One has~\mbox{$\dim Z\le n-1$}, so that $G$ has a fixed point on $Z$.
\end{proof}

\section{Surfaces}
\label{section:surfaces}

In this section we collect some facts about $p$-groups acting on surfaces.
In particular, we prove Propositions~\ref{proposition:Cr2} and~\ref{proposition:Cr2-fixed-point}.

\begin{lemma}\label{lemma:dP}
Let $S$ be a smooth del Pezzo surface, and $G\subset\Aut(S)$
be a $p$-group.
Suppose that $p\ge 5$. Then $G$ has a fixed point on $S$.
\end{lemma}
\begin{proof}
If $S\cong\P^2$, then $G$ has a fixed point
on $S$ by Lemma~\ref{lemma:PGL}.
If~\mbox{$S\cong\P^1\times\P^1$},
then $G$ has a fixed point on $S$ by Corollary~\ref{corollary:P1xP1}.
Therefore, we may assume that $S$ is obtained from
$\P^2$ by blowing up $1\le r\le 8$ points.
Put~\mbox{$d=K_S^2=9-r$}.

If $d=1$, then $G$ fixes the unique base point of the linear system
$|-K_S|$. If $d=2$, then the anticanonical linear system gives a $G$-equivariant double cover
$\phi\colon S\to\P^2$. We know from Lemma~\ref{lemma:PGL} that there is a $G$-invariant
point $P\in\P^2$, so that the fiber $\phi^{-1}(P)$
consists of $G$-invariant points on~$S$.
Thus we will assume that $3\le d\le 8$.
We know the number of $(-1)$-curves on~$S$:
see e.g.~\cite[Theorem~IV.4.3(c)]{Manin-CubicForms}
for degrees $d\le 6$. If this number is not divisible by $p$,
then there is a $G$-invariant $(-1)$-curve $C\cong\P^1$ on~$S$,
and $G$ has a fixed point on $C$ by Lemma~\ref{lemma:PGL}.
This happens unless~\mbox{$p=5$} and $d=5$.
In the latter case one has $\Aut(S)\cong\mathfrak{S}_5$, see e.g.~\cite[Theorem~8.5.8]{Dolgachev-ClassicalAlgGeom}.
This means that $G$ is a cyclic group, and $G$ has a fixed point on $S$
by the holomorphic Lefschetz fixed-point formula.
\end{proof}

\begin{lemma}\label{lemma:smooth-rational-surface}
Let $S$ be a smooth rational surface, and $G\subset\Aut(S)$ be a $p$-group.
Suppose that $p\ge 5$. Then $G$ has a fixed point on $S$.
\end{lemma}
\begin{proof}
Let $\pi\colon S\to S'$ be a result of a $G$-Minimal Model Program
ran on~$S$. Then either $S'$ is a del Pezzo surface,
or there is a $G\QQ$-conic bundle structure~\mbox{$\phi\colon S'\to \P^1$} (see \cite[Theorem~1G]{Iskovskikh-1979s-e}).
In the former case $G$ acts on~$S'$
with a fixed point by Lemma~\ref{lemma:dP}.
In the latter case $G$ has a fixed point on~$S'$ by Corollary~\ref{corollary:MFS}
and Corollary~\ref{corollary:rational-curve}.
Now Lemma~\ref{lemma:G-MMP} implies that $G$ has a fixed point on $S$.
\end{proof}

Now we are ready to prove Proposition~\ref{proposition:Cr2-fixed-point}.

\begin{proof}[Proof of Proposition~\xref{proposition:Cr2-fixed-point}]
The minimal resolution of singularities of $S$ is $G$-equivariant.
Keeping in mind that an image of a $G$-fixed point with respect
to any $G$-equivariant morphism is again a $G$-fixed point, we may assume
that $S$ is smooth. Now the assertion follows from
Lemma~\ref{lemma:smooth-rational-surface}.
\end{proof}

Finally, we prove Proposition~\ref{proposition:Cr2}.

\begin{proof}[Proof of Proposition~\xref{proposition:Cr2}]
Regularizing the (rational) action of $G$
(see \cite[Lemma-Definition~3.1]{Prokhorov-Shramov-2013}),
we may assume that $G$ acts by automorphisms of a smooth rational surface $S$.
By Proposition~\ref{proposition:Cr2-fixed-point}
(or by Lemma~\ref{lemma:smooth-rational-surface})
there is a $G$-fixed point
on $S$. Now everything follows from Remark~\ref{remark:Aut-P} and Lemma~\ref{lemma:GL}.
\end{proof}

\begin{remark}
Theorem~\ref{theorem:Beauville}
shows that the assertions of Propositions~\ref{proposition:Cr2}
and~\ref{proposition:Cr2-fixed-point} fail without
the assumption~\mbox{$p\ge 5$}.
See also the examples of~\S\ref{section:examples}.
\end{remark}

Applying Proposition~\ref{proposition:Cr2-fixed-point} together with Corollary~\ref{corollary:MFS},
we obtain the following result.

\begin{corollary}\label{corollary:MFS-dim-3}
Let $X$ be a rationally connected threefold, and let~\mbox{$G\subset\Aut(X)$} be a $p$-group.
Suppose that $\phi\colon X\to S$ is a $G$-Mori fiber space with~\mbox{$\dim S>0$}.
Suppose that $p\ge 5$.
Then $G$ has a fixed point on $X$.
\end{corollary}

We conclude this section by two useful facts about $p$-groups acting on non-rational
surfaces.

\begin{lemma}\label{lemma:K3}
Let $S$ be a $K3$ surface with at worst Du Val singularities,
and let~\mbox{$G\subset\Aut(S)$} be a $p$-group.
Suppose that $p\ge 5$.
Then $G$ has a fixed point on~$S$.
\end{lemma}
\begin{proof}
Replacing $S$ with its minimal resolution we may assume that $S$ is smooth.
Let $G_\mathrm{s}$ be the subgroup of $G$ that acts on $S$
by symplectic automorphisms, i.e. $G_{\mathrm{s}}$ is
the kernel of the induced action of $G$ on $H^{2,0}(S)$
(cf. \cite[Definition 0.2]{Nikulin-1980-aut}).
Then $G/G_{\mathrm{s}}$ is a cyclic group.
If $G_\mathrm{s}$ is trivial, then~$G$ has a fixed point on $S$
by the holomorphic Lefschetz fixed-point formula.
Thus we assume that $G_\mathrm{s}$ is non-trivial.

Suppose that $|G_{\mathrm{s}}|>p$. Then $G_{\mathrm{s}}$ contains a subgroup
$\hat{G}_{\mathrm{s}}$ of order $p^2$ by Lemma~\ref{lemma:basic}(i).
The group $\hat{G}_{\mathrm{s}}$ is abelian, which is impossible by~\cite[Theorem 4.5]{Nikulin-1980-aut}. Therefore,
one has $G_{\mathrm{s}}\cong\mumu_p$.
Moreover, it has exactly~\mbox{$24/(p+1)$}
fixed points on $S$ (see \cite[\S5.1]{Nikulin-1980-aut}).
Since $24/(p+1)<p$ for $p\ge 5$,
these points cannot be permuted by $G/G_{\mathrm{s}}$.
Thus $G$ has a fixed point on~$S$.
\end{proof}

\begin{lemma}\label{lemma:ruled-surface}
Let $S$ be a birationally ruled surface
over an elliptic curve, and $G\subset\Bir(S)$ be a $p$-group.
Suppose that $p\ge 5$. Then~\mbox{$\rr(G)\le 3$}.
\end{lemma}
\begin{proof}
There exists a rational map $\phi\colon S\dasharrow E$, where $E$ is an elliptic curve
and a general fiber of $\phi$ is a rational curve. Moreover, $\phi$ is equivariant with
respect to the group $\Bir(S)$. Thus we have an exact sequence
$$
1\to \Gamma_{\phi}\to \Bir(S)\to \Gamma_S\to 1,
$$
where $\Gamma_{\phi}$ acts by fiberwise birational transformations with respect to $\phi$,
and~\mbox{$\Gamma_S\subset\Bir(E)=\Aut(E)$}.
Let $\mathscr{S}$ be the fiber of $\phi$ over the general scheme-theoretic point of $E$.
Then $\mathscr{S}$ is a rational curve over the function
field~\mbox{$\Bbbk(E)$}, so that
$\Gamma_{\phi}\subset\PGL_2\big(\Bbbk(E)\big)$ is a cyclic group. On the other hand,
the elements of $\Gamma_S$ act by translations on $E$, so that
$\Gamma_E$ is generated by at most two elements, and the assertion follows.
\end{proof}

\section{Anticanonical divisors}
\label{section:anticanonical}

In this section we study invariant anticanonical divisors
with respect to $p$-groups acting on Fano threefolds.

\begin{lemma}\label{lemma:invariant-member}
Let $X$ be a terminal Fano threefold, and~\mbox{$G\subset\Aut(X)$} be a $p$-group
such that $X$ is a $G\QQ$-Fano threefold.
Suppose that $G$ has no fixed points on~$X$. Furthermore, suppose that
there exists a $G$-invariant divisor~\mbox{$S\in |-K_X|$}, and that~\mbox{$p\ge 5$}.
Then the pair $(X,S)$ is log canonical and
one of the following cases occurs.
\begin{enumerate}
\renewcommand\labelenumi{(\Alph{enumi})}
\renewcommand\theenumi{(\Alph{enumi})}
\item
The surface $S$ is reducible \textup(and reduced\textup).
The group $G$ acts transitively on the set of irreducible components of~$S$.
\item
The surface $S$ is irreducible and not normal.
Then $S$ is a birationally ruled surface over an elliptic curve.
Let $\Lambda\subset S$ be the non-normal locus and let $\nu\colon S'\to S$
be the normalization. Then $\nu^{-1}(\Lambda)$ is either
a smooth
elliptic curve or a disjoint union of two smooth elliptic curves.
Moreover, $S'$ has at worst Du Val singularities and
$\nu^{-1}(\Lambda)$ is contained in the smooth locus of~$S'$.
Furthermore, $\Lambda$ has at most two irreducible components,
and each of them is a $G$-invariant curve of geometric genus~$1$.
\end{enumerate}
In both cases the action of $G$ on $S$ is faithful.
In case~\textup{(B)} one has $\rr(G)\le 3$.
\end{lemma}
\begin{proof}
First we claim that for any effective $G$-invariant
$\QQ$-divisor $D$ on~$X$ such that $-(K_X+D)$ is nef
the pair $(X,D)$ is log canonical.
Indeed, assume the converse. Then replacing $D$ with $cD$
for some rational number~\mbox{$0<c<1$}
we may assume that
the pair $(X,D)$ is strictly log canonical
and~\mbox{$-(K_X+D)$} is ample.
Let $Z \subset X$ be a minimal $G$-center of non Kawamata log terminal singularities of
the pair $(X, D)$, see~\cite[\S2]{ProkhorovShramov-RC}.
By \cite[Lemma 2.2]{ProkhorovShramov-RC}
there exists a $G$-invariant effective $\QQ$-divisor $D'$ such that
the pair $(X,D')$ is strictly log canonical, the divisor $-(K_X+D')$ is ample,
and the only centers of non Kawamata log terminal singularities of
$(X, D')$ are the irreducible
components of $Z$.
By Shokurov's connectedness theorem
\cite[17.4]{Utah} the subvariety $Z$ is connected.
If $\dim Z=0$, then $Z$ is a $G$-fixed point.
This contradicts our assumptions.
Thus $\dim Z>0$.
By Kawamata's subadjunction theorem $Z$ is normal
and there exists an effective $\QQ$-divisor~$\Delta$ on~$Z$ such that
the pair $(Z, \Delta)$ is Kawamata log terminal and~\mbox{$-(K_Z+\Delta)$} is ample.
In particular, $Z$ is a rational curve.
By Corollary~\ref{corollary:rational-curve}
there is a $G$-fixed point on $Z$, which is a contradiction.
This shows that $(X,D)$ is log canonical.
In particular, $(X,S)$ is log canonical.

Suppose that $S$ is reducible.
Write $S=\sum S_i$. Let $S_1\subset S$ be an irreducible component,
let $S_1,\dots, S_r$ be its $G$-orbit, and let~\mbox{$S'=S_1+\ldots+S_r$}.
If $S'\neq S$, then $-(K_X+cS')$ is ample for some $c>1$.
Since the pair~\mbox{$(X,cS')$} is not log canonical, this is impossible.
Hence $S=S'$, i.e. $G$ acts transitively
on the set of irreducible components of $S$.
This is case~(A) of the lemma.

From now on we assume that $S$ is irreducible. If $S$ is
normal and its singularities are worse than Du Val, then
$S$ has at most two non-Du Val points by~\cite[Theorem~6.9]{Shokurov-1992-e-ba}.
Hence $G$ has a fixed point on $S$ in this case.
If $S$ has at worst Du Val singularities, then $G$ has a fixed point on $S$
by Lemma~\ref{lemma:K3}.
Finally assume that $S$ is not normal.
Let~\mbox{$\nu\colon S'\to S$} be its normalization. We have the adjunction formula
\[
K_{S'}+\Delta'\sim\nu^* K_S\sim 0,
\]
where $\Delta'$ is \emph{the different},
an effective integral Weil divisor such that $\nu(\Delta')$
is supported on the
non-normal locus. Moreover, $\Delta'$ is $G$-invariant and the pair $(S',\Delta')$ is log canonical
\cite{Kawakita2007}.
Hence $S'$ is a birationally ruled surface.
Again by~\cite[Theorem~6.9]{Shokurov-1992-e-ba}
the locus $\mathcal{L}$ of log canonical singularities of $(S',\Delta')$ has at most two connected components.
Since $G$ has no fixed points on $S$, and hence on $S'$,
the surface $S'$ is not rational
by Proposition~\ref{proposition:Cr2-fixed-point}. Moreover, each connected component of
$\mathcal{L}$ is one-dimensional.
In particular, all zero-dimensional centers of log canonical singularities
of~\mbox{$(S',\Delta')$} are contained in $\Delta'$, so that $S'$ has only log terminal singularities.
Furthermore, since~\mbox{$K_{S'}+\Delta'\sim 0$}, the singularities of $S'$ outside of $\Delta'$ are at
worst Du Val.
Thus the Abanese map gives a morphism $\varphi\colon S'\to Z$
to a non-rational curve $Z$. Let $\Delta_1'\subset \Delta'$ be a $\varphi$-horizontal component.
By the adjunction formula the divisor
$-K_{\Delta_1'}$ is effective. Thus~\mbox{$p_a(\Delta_1')\le 1$}. Since $\Delta_1'$
is not rational, it is a smooth
elliptic curve. Moreover, again by adjunction
$\Delta_1'$ is contained in the smooth locus of $S'$ and does not meet other
components of $\Delta'$. Clearly, $\Delta_1'$ is $G$-invariant, and hence so is $\nu(\Delta_1')$.
Since $G$ has no fixed points on $S$, the curve $\nu(\Delta'_1)$ cannot be rational by Corollary~\ref{corollary:rational-curve}.
Hence it is an elliptic curve and the
restriction~\mbox{$\varphi_{\Delta_1'}\colon\Delta_1' \to Z$} is \'etale.
If $\Delta'$ is connected, then $\Delta'=\Delta'_1$.
Hence the non-normal locus of~$S$ coincides with~\mbox{$\nu(\Delta'_1)$}.

Finally consider the case when $\Delta'$ is not connected.
Then by \cite[Theorem~6.9]{Shokurov-1992-e-ba} it has two connected components
and both of them are sections of $\varphi$. Moreover, both
are $G$-invariant. Arguing as above, we see that they are smooth
elliptic curves, and the non-normal locus of $S$ is a union of their images
under the morphism~$\nu$.

We see that the non-normal locus $\Lambda\subset S$ is a union of at most
two irreducible $G$-invariant curves.
Moreover, by Corollary~\ref{corollary:rational-curve} none
of them can be a rational curve, so that their geometric genus
equals~$1$. This completes a description of case~(B) of the lemma.

In either of the cases (A) or (B) we see that $S$ has a one-dimensional
singular locus. On the other hand, the threefold $X$ has isolated singularities.
Suppose that $\gamma\in G$ is an element that acts trivially on $S$.
Let $P$ be a point that is singular on $S$ but non-singular on $X$.
Then $T_P(S)=T_P(X)$. The action of $\gamma$ on $T_P(S)$ is
trivial, while the action of $\gamma$ on $T_P(X)$ is non-trivial by
Remark~\ref{remark:Aut-P}. An obtained contradiction shows that
the action of $G$ on $S$ is faithful.

Finally, applying Lemma~\ref{lemma:ruled-surface}
we see that in case~(B) one has~\mbox{$\rr(G)\le 3$}.
\end{proof}

\section{Gorenstein Fano threefolds}
\label{section:Gorenstein}

In this section we study $p$-groups acting on Gorenstein Fano threefolds.

Let $X$ be a Fano threefold
with at worst canonical Gorenstein singularities.
In this case, the number
\begin{equation*}
\g(X)=\frac{1}{2}(-K_X)^3+1
\end{equation*}
is called the \textit{genus} of $X$.
By Riemann--Roch theorem and Kawamata--Viehweg vanishing one has
\begin{equation*}
\dim |-K_X|=\g(X)+1
\end{equation*}
(see e.\,g. \cite[2.1.14]{Iskovskikh-Prokhorov-1999}).
In particular, the genus $\g(X)$ is an integer, and~\mbox{$\g(X)\ge 2$}.
The maximal number $\iota=\iota(X)$
such that $-K_X$ is divisible by~$\iota$ in~\mbox{$\Pic(X)$} is called the \textit{Fano index}, or sometimes just \emph{index}, of~$X$.
Recall that~\mbox{$\Pic(X)$} is a finitely generated torsion free abelian group,
see e.\,g. ~\mbox{\cite[Proposition 2.1.2]{Iskovskikh-Prokhorov-1999}}.
The rank $\rho(X)$ of the free abelian group~\mbox{$\Pic(X)$} is called the \emph{Picard rank} of~$X$.
Let $H$ be a divisor class such that~\mbox{$-K_X\sim\iota(X) H$}.
The class $H$ in $\Pic(X)$ is unique since $\Pic(X)$ is torsion free.
Define the \textit{degree} of $X$ as~\mbox{$\dd(X)=H^3$}.
Since the class of~$H$ is $\Aut(X)$-invariant, the rational map
$X\dasharrow\P^N$ given by a linear system~\mbox{$|kH|$} is always
$\Aut(X)$-equivariant.

\begin{lemma}\label{lemma:GQ-Fano-rk-Cl}
Let $X$ be a $G\QQ$-Fano variety, where $G$ is a $p$-group.
Then either $\rk\Cl(X)=1$, or $\rk\Cl(X)\ge p$.
\end{lemma}
\begin{proof}
Suppose that $\rk\Cl(X)>1$.
Consider the representation of $G$ in the $\QQ$-vector space $W=\Cl(X)\otimes\QQ$.
There is a one-dimensional subrepresentation $K\subset W$ corresponding to
the canonical class of $X$. Put~\mbox{$V=W/K$}. Then $V$ is a (possibly non-faithful) $G$-representation
defined over $\QQ$. Since~\mbox{$\rk\Cl(X)^G=1$}, the representation $V$ has no trivial subrepresentations.
By Lemma~\ref{lemma:basic}(iii) we have $\dim V\ge p-1$, so that~\mbox{$\dim W=\rk\Cl(X)\ge p$}.
\end{proof}

\begin{lemma}\label{lemma:rho-large}
Let $X$ be a terminal Gorenstein
Fano threefold
with~\mbox{$\rho(X)>1$}, and let~\mbox{$G\subset\Aut(X)$} be a $p$-group
such that $X$ is a $G$-Fano variety. Then~\mbox{$p\le 3$}.
\end{lemma}
\begin{proof}
By Lemma~\ref{lemma:GQ-Fano-rk-Cl} we have $\rho(X)\ge p$.
On the other hand, by \cite{Prokhorov-GFano-2} one has $\rho(X)\le 4$, and the assertion follows.
\end{proof}

\begin{lemma}\label{lemma:rho-1-iota-large}
Let $X$ be a terminal Gorenstein Fano threefold with~\mbox{$\rho(X)=1$}
and~\mbox{$\iota(X)>1$}.
Let $G\subset\Aut(X)$ be a $p$-group such that $X$ is a $G$-Fano variety. Suppose that
$p\ge 5$. Then $G$ has a fixed point on $X$.
\end{lemma}
\begin{proof}
It is known that $\iota(X)\le 4$. Moreover, $\iota(X)= 4$
if and only if $X$ is the projective space~$\P^3$, and~\mbox{$\iota(X)=3$}
if and only if $X$ is a quadric in~$\P^4$ (see e.\,g.
\mbox{\cite[Theorem~3.9]{Shin1989}}).
In the former case the assertion is implied by Lemma~\ref{lemma:PGL},
and in the latter case the assertion follows from Lemma~\ref{lemma:proj-Heisenberg}.

Thus we may assume that $\iota(X)=2$. Recall that $1\le \dd(X)\le 5$
(see e.\,g. \cite[Corollary 8.7]{Prokhorov-GFano-1}).

If $\dd(X)=5$, then $X$ is smooth \cite[Corollary 8.7]{Prokhorov-GFano-1}.
Thus $X$ is isomorphic to a linear section
of the Grassmannian~\mbox{$\operatorname{Gr}(2,5)\subset\P^9$} by a
subspace~\mbox{$\P^6\subset\P^9$}, see \cite[\S12.2]{Iskovskikh-Prokhorov-1999}.
In this case one has
\begin{equation*}
\Aut(X)\cong \PGL_2(\Bbbk),
\end{equation*}
see~\mbox{\cite[Proposition~4.4]{Mukai-1988}} or~\cite[Proposition~7.1.10]{CheltsovShramov2016}.
Hence $G$ is a cyclic group and the assertion follows by
the holomorphic Lefschetz fixed point formula.

If $\dd(X)=4$, then
$X$ is a complete intersection of two quadrics in $\P^5$
(see e.\,g. \cite[Corollary 0.8]{Shin1989}).
Thus the assertion follows from Lemma~\ref{lemma:complete-int-of-two-quadrics-or-2-3}.

If $\dd(X)=3$, then $X\cong X_3\subset \P^4$ is a cubic threefold
(see e.\,g. \cite[Corollary 0.8]{Shin1989}).
Thus the assertion follows from Lemma~\ref{lemma:proj-Heisenberg}.

If $\dd(X)=2$, then the linear system $|-\frac 12 K_X|$
defines an $\Aut(X)$-equivariant double cover~\mbox{$X\to \PP^3$}.
By Lemma~\ref{lemma:PGL} the group $G$
has a fixed point on $\PP^3$, and hence also on $X$,
cf. Lemma~\ref{lemma:double-cover-1} below.

Finally, if $\dd(X)=1$,
then the linear system $|-\frac 12 K_X|$ has a unique base point
(see \cite[Theorem~0.6]{Shin1989})
which must be fixed by $\Aut(X)$.
\end{proof}

\begin{lemma}[{cf. \cite[Lemma 5.3]{Prokhorov2009e}, \cite[Lemma~4.4.1]{Kuznetsov-Prokhorov-Shramov},
\cite[Proposition~6.1.1]{Prokhorov-Shramov-JCr3}}]
\label{lemma:double-cover-1}
Let $X$ be a Fano threefold with canonical Gorenstein singularities,
and let $G\subset\Aut(X)$ be a $p$-group.
Suppose that $-K_X$ is not very ample, and that $p\ge 5$.
Then $G$ has a fixed point on $X$.
\end{lemma}

\begin{proof}
Suppose that $G$ has no fixed points on $X$.
Recall that one has~\mbox{$\dim H^0(X,-K_X)=\g(X)+2$}.
If the linear system $|-K_X|$
is not base point free, then $\Bs |-K_X|$ is either a single point or a rational curve
\cite[Theorem 0.5]{Shin1989}. This contradicts
our assumption that $G$ has no fixed points on $X$,
see Corollary~\ref{corollary:rational-curve}.
Thus the linear system $|-K_X|$ defines a morphism
\[
\Phi=\Phi_{|-K_X|}\colon X\to Y\subset \PP^{2\g(X)-2}.
\]
If $\Phi$ is birational, then standard inductive arguments \cite[Lemma 2.2.5]{Iskovskikh-Prokhorov-1999}
show that
it is an embedding and $-K_X$ is very ample,
a contradiction. So assume that $\Phi$ is not birational.
Then by \cite[Proposition 2.1.15]{Iskovskikh-Prokhorov-1999}
the morphism $\Phi$ is a double cover and
$Y=Y_{\g(X)-1}\subset \PP^{2\g(X)-2}$ is a subvariety of minimal degree (see e.g. \cite[2.2.11]{Iskovskikh-Prokhorov-1999}).
Since $\Phi$ is $\Aut(X)$-equivariant, it is sufficient to show that $G$ has a fixed point on $Y$.
If $Y$ is singular, then the singular locus is either a single point or a line
and we are done. Assume that $Y$ is smooth.
Then it is either $\PP^3$, or a quadric in $\PP^4$,
or a scroll over $\PP^1$, and the existence of a fixed point
follows from Lemma \ref{lemma:rho-1-iota-large} and
Corollary~\ref{corollary:MFS}.
\end{proof}

\begin{lemma}\label{lemma:Fano-3-fold-small-g}
Let $X$ be a terminal Gorenstein Fano threefold.
Suppose that $\rho(X)=1$, $\iota(X)=1$, and $\g(X)\le 5$.
Let $G\subset\Aut(X)$ be a $p$-group. Suppose that
$p\ge 5$.
Then $G$ has a fixed point on $X$.
\end{lemma}
\begin{proof}
By Lemma~\ref{lemma:double-cover-1}
we may assume that \mbox{$-K_X$} is very ample and so $\g(X)\ge 3$.
If $\g(X)=3$,
then~$X$ is a quartic in $\P^4$
(because $\dim |-K_X|=4$ and~\mbox{$-K_X^3=4$}),
so that the assertion follows from Lemma~\ref{lemma:proj-Heisenberg}.
If $\g(X)=4$, then $X$ is a complete intersection of a quadric and a cubic in
$\P^5$ (see \cite[Proposition 4.1.12]{Iskovskikh-Prokhorov-1999}).
In this case the assertion follows from Lemma~\ref{lemma:complete-int-of-two-quadrics-or-2-3}.
Finally, if $\g(X)=5$, then $X$ is a complete intersection
of three quadrics in~$\P^6$ (see \cite[Theorem 1.6]{Przhiyalkovskij-Cheltsov-Shramov-2005en}).
In this case the assertion follows from Lemma~\ref{lemma:three-quadrics}.
\end{proof}

\begin{lemma}\label{lemma-g-equiv1}
Let $X$ be a terminal Gorenstein Fano threefold with $\rho(X)=1$
and~\mbox{$\iota(X)=1$}.
Let $G\subset\Aut(X)$ be a $p$-group such that $X$ is a $G$-Fano variety.
Suppose that $\g(X)\ge 6$ and $p\ge 5$.
Suppose also that $\g(X)\not\equiv 1\mod p$.
Then $X$ is $\QQ$-factorial.
\end{lemma}
\begin{proof}
Suppose that $X$ is not $\QQ$-factorial, so that $\rk\Cl(X)>1$.
Let $S$ be an effective Weil divisor on $X$ which is not $\QQ$-Cartier.
Take $S$ to be of minimal possible degree.
Let $S_1=S, \ldots, S_N$ be its $G$-orbit. Then $\sum S_i\sim -aK_X$
for some $a\in \ZZ$. Clearly, $N=p^k$ for some $k$.
Therefore, one has
\[
p^k\deg (S) =(2\g(X)-2)a.
\]
Since $2\g(X)-2$ is coprime to $p$, we can write $a=p^kb$ and
\[
\deg (S) =(2\g(X)-2)b
\]
for some $b\in\mathbb{Z}$.
In particular, $X$ does not contain surfaces of degree less
than~\mbox{$2\g(X)-2$}.
Now let $\xi\colon X'\to X$
be a $\QQ$-factorialization. Run the Minimal Model Program on $X'$.
We get the following diagram:
\[
\xymatrix{
&X'\ar[dl]_{\xi}\ar@{-->}[rr]^{\psi}&&Y\ar[dr]^f
\\
X&&&&Z
}
\]
Here $\psi\colon X'\dasharrow Y$ is a
birational
map given by the Mininal Model Program,
and
$f\colon Y\to Z$ is a Mori fiber space.
By \cite[Theorem 1.1]{Prokhorov-planes} does not contain planes,
and so each step of the Mininal Model Program
is a divisorial contraction of threefolds with terminal Gorenstein singularities
\cite[Proposition 4.5]{Prokhorov-2005a}.
It is easy to check that
on each step of the Minimal Model Program we have
\[
(-K_{X_i})^2\cdot F\ge 2\g(X)-2\ge 10
\]
for any surface $F\subset X_i$ (see \cite[Lemma 2.5]{Prokhorov-v22}). Hence each step contracts
a divisor to a curve (otherwise the anticanonical degree
of the exceptional divisor is at most $4$,
see \cite[Proposition 2.2]{Prokhorov-v22}).
In particular, $\psi$ is a morphism,
all varieties $X_i$ are Gorenstein, and thus $Y$ is Gorenstein as well, see~\cite{Cutkosky-1988}.
Note that $\xi$ contracts \emph{all} curves of degree $0$ with respect to~$K_{X'}$.
Hence $-K_{X'}$ is ample over $Z$.
Assume that $Z$ is a curve. Then a general fiber
$F$ of the composition
\[
\pi=f\circ\psi\colon X'\to Z
\]
is a smooth del Pezzo surface.
Therefore, we have $K_{X'}^2\cdot F=K_F^2\le 9$, a contradiction.
Now assume that $Z$ is a surface.
It is smooth by \cite{Cutkosky-1988}.
By the above argument with a degree of a fiber
we may assume that $Z$ has no contractions to a curve,
and so $Z\cong \PP^2$.
We claim that $\pi\colon X'\to Z$ has no two-dimensional fibers.
Indeed, assume that there
is a two-dimensional irreducible component of a fiber.
Note that by our assumptions running the Minimal Model Program
over $Z$ we cannot obtain a Mori fiber space
over a surface dominating $Z$. Hence on some step we get a model
$\hat{f}\colon \hat X\to Z$ which is an equidimensional (but possibly non-standard) conic bundle, and the last
contraction~\mbox{$\check X\to \hat X$}
that brings us to $\hat{X}$ contracts a divisor
$\check F$ to an irreducible component $\hat B$ of a fiber of~$\hat{f}$.
Clearly, $\hat{B}$ is a smooth rational curve, and $-K_{\hat X}\cdot \hat{B}$ equals $1$ or $2$.
Thus $-K_{\check X}\cdot \check F\le 4$
(see e.g. \cite[Lemma 2.4]{Prokhorov-v22}).
This contradicts our assumptions. Hence $\pi\colon X'\to Z=\PP^2$ is an equidimensional conic bundle.

Let $\Delta\subset Z$ be the
discriminant curve of $\psi$, let $l\subset \PP^2$ be a general line,
and put $F=\psi^{-1}(l)$. Then $F$ is a smooth rational surface with
a conic bundle structure $\theta\colon F\to l\cong \PP^1$.
The number of degenerate fibers of~$\theta$
equals~\mbox{$l\cdot \Delta=\deg \Delta$}.
Therefore, we compute
\[
K_{X'}^2\cdot F= (K_{X'}+F)^2\cdot F-2K_{X'}\cdot F^2
=K_F^2+4=10-\rho(F)+4=12-\deg \Delta.
\]
On the other hand, we have
\[
K_{X'}^2\cdot F =\deg \xi(F)=(2\g(X)-2)b
\]
for some $b\in\mathbb{Z}$. We get the following possibilities:
\begin{itemize}
\item
$\deg \xi(F)=12$, $\Delta=\varnothing$, $\g(X)=7$,
\item
$\deg \xi(F)=10$, $\deg \Delta=2$, $\g(X)=6$.
\end{itemize}
Note that $\rho(X')=\rk\Cl(X)\ge p\ge 5$
by Lemma~\ref{lemma:GQ-Fano-rk-Cl},
and so $\rho(X'/Z)\ge 4$.
Hence $\Delta$ has at least $4$ irreducible components.
This is a contradiction.

Therefore, $f$ cannot be of fiber type, i.e. $Z$ is a point.
Hence $Y$ is a (terminal) Gorenstein Fano threefold with $\rho(Y)=1$.
Note that each step increases the degree $-K_{X_i}^3$ by at least $4\g(X)-6$
(see e.\,g. \cite[Proposition~5.1]{Prokhorov-planes})
and
$$\rho(X')=\rk\Cl(X)\ge p\ge 5$$
by Lemma~\ref{lemma:GQ-Fano-rk-Cl}.
Thus
\[
-K_Y^3\ge 2\g(X)-2+ (p-1)(4\g(X)-6)\ge 82,
\]
which gives a contradiction with \cite{Prokhorov-2005a}.
\end{proof}

\begin{lemma}\label{lemma:g-12}
Let $X$ be a terminal Gorenstein Fano threefold with $\rho(X)=1$, \mbox{$\iota(X)=1$}, and $\g(X)=12$.
Let $G\subset\Aut(X)$ be a $p$-group such that $X$ is a $G$-Fano variety.
Suppose that $p\ge 3$. Then $G$ has a fixed point on $X$.
\end{lemma}
\begin{proof}
Since $X$ is a $G$-Fano variety, by Lemma~\ref{lemma:GQ-Fano-rk-Cl} one has either~\mbox{$\rk\Cl(X)=1$}, or
$\rk\Cl(X)\ge p>2$.
Thus $\rk\Cl(X)=1$ and $X$ is smooth by \cite[Theorem 1.3]{Prokhorov-v22}.
Hence the family $\mathscr{C}$ parameterizing the conics (in the anticanonical embedding)
on $X$ is isomorphic to the projective plane $\PP^2$,
see \cite[Theorem 2.4]{Kollar2004b}, \cite[Proposition B.4.1]{Kuznetsov-Prokhorov-Shramov}.
By Lemma~\ref{lemma:PGL} the group~$G$ has a fixed point on $\mathscr{C}$,
which means that there is a $G$-invariant conic contained in $X$.
Now the assertion follows from Corollary~\ref{corollary:conic}
\end{proof}

\begin{lemma}\label{lemma:g-ge-9}
Let $X$ be a terminal
Gorenstein Fano threefold with $\rho(X)=1$, \mbox{$\iota(X)=1$}, and $\g(X)\ge 9$.
Let $G\subset\Aut(X)$ be a $p$-group such that $X$ is a $G$-Fano variety.
Suppose that $p\ge 5$.
Then $G$ has a fixed point on $X$.
\end{lemma}
\begin{proof}
By Lemma~\ref{lemma:g-12} we may assume that $\g(X)\neq 12$, so that either~\mbox{$\g(X)=9$}, or $\g(X)=10$.
Thus the threefold $X$ is $\QQ$-factorial
by Lemma~\ref{lemma-g-equiv1}.

Assume that $X$ is singular.
Since $X$ is $\QQ$-factorial, there are at most $3$ singular points on $X$ by \cite[Theorem 1.3]{Prokhorov-planes}. Since $p>3$, these points must be $G$-invariant.

Therefore, we may assume that $X$ is smooth.
Let $\gamma\in G$ be a non-trivial element, and let
$\Lambda=X^{\gamma}$ be its fixed locus.
By the holomorphic Lefschetz fixed point formula $\Lambda$ is not empty.
Thus we may assume that $G\neq \langle \gamma\rangle$, i.e. the group~$G$ is not cyclic.
There is a natural action of $G$ on $H^{3}(X,\ZZ)$.
Recall that
\[
\rk H^{3}(X,\ZZ)= 2\dim H^{1,2}(X)=2(12-\g(X))\le 6
\]
for $\g(X)\ge 9$, see e.\,g. \cite[\S12.2]{Iskovskikh-Prokhorov-1999}.
Replacing $\gamma$ with another element of~\mbox{$G\setminus \{1\}$} if necessary,
we may assume that one of its eigenvalues on $H^{3}(X,\CC)$ equals $1$.
In this case either the action of $\gamma$ on $H^{3}(X,\ZZ)$ is trivial, or one has
$p=5$, $\g(X)=9$, $\rk H^{3}(X,\ZZ)=6$, and $\operatorname{tr}_{H^{3}(X,\ZZ)} \gamma^*=1$,
see Lemma~\ref{lemma:basic}(iii).

If $\dim \Lambda>0$, then $\Lambda$ meets a line $l\subset X$.
Let $P$ be one of the points of~\mbox{$l\cap\Lambda$}.
Since $X$ is an intersection of quadrics \cite[Corollary 4.1.13]{Iskovskikh-Prokhorov-1999},
the number of lines on $X$ passing through $P$ is at most~$4$
(see \cite[Proposition~4.2.2(iv)]{Iskovskikh-Prokhorov-1999}).
Since $p\ge 5$, this implies that these lines, and in particular the line $l$,
are $\gamma$-invariant.

As in \cite[Proof of Lemma 5.18]{Prokhorov2009e}, considering
the double projection from $l$ we get a $\gamma$-equivariant Sarkisov link
$X \dashrightarrow Y$, where $Y\cong\PP^3$ in the case $\g(X)=9$ and
$Y$ is a smooth quadric in $\PP^4$ in the case $\g(X)=10$
(see e.g. \cite[Theorem 4.3.3]{Iskovskikh-Prokhorov-1999}).
Moreover, $X \dashrightarrow Y$ contracts a surface to a smooth curve $\Gamma \subset Y$
with $H^3(X,\ZZ)\cong H^1(\Gamma,\ZZ)$.
Since $X \dashrightarrow Y$ is birational, the action of $\gamma$ on $Y$
is not trivial.
The curve $\Gamma$ is of degree~$7$, it spans $\PP^3$ (respectively, $\PP^4$),
and $\g(\Gamma)=12-\g(X)$.
Since $\Gamma$ spans $\PP^3$ (respectively, $\PP^4$) the action of $\gamma$ on $\Gamma$
is non-trivial.
By the topological Lefschetz fixed point formula
it is non-trivial
on~\mbox{$H^3(X,\ZZ)\cong H^1(\Gamma,\ZZ)$},
and so one has $p=5$, $\g(X)=9$, and~\mbox{$\g(\Gamma)=3$}.
This is impossible by Lemma~\ref{lemma:g-3-deg-7}.

Therefore, one has $\dim \Lambda=0$.
Again by the topological Lefschetz fixed point formula this gives
\[
\operatorname{tr}_{H^{3}(X,\ZZ)} \gamma^*<4\le \dim H^{1,2}(X),
\]
and so $\operatorname{tr}_{H^{3}(X,\ZZ)} \gamma^*=1$,
i.e. $\gamma$ has exactly $3$ fixed points on $X$.
Moreover, we may assume that the action of $G$ on
$H^{3}(X,\ZZ)$ is faithful; otherwise we replace $\gamma$ with another
element of $G$ which acts trivially on~\mbox{$H^3(X,\ZZ)$}, and return to the case
of a positive-dimensional fixed locus~$\Lambda$.
Hence $G$ is abelian by Lemma~\ref{lemma:basic}(v).
Therefore, the fixed points of $\gamma$ are fixed by~$G$.
\end{proof}

\begin{lemma}\label{lemma:KX-irreducible-representation}
Let $X$ be a terminal Gorenstein
Fano threefold such that the divisor~\mbox{$-K_X$}
is very ample. Let $G\subset\Aut(X)$ be a $p$-group.
Then the representation of $G$ in
$H^0(X,-K_X)$ is reducible.
\end{lemma}
\begin{proof}
Suppose that $H^0(X,-K_X)$ is an irreducible representation of $G$.
By Schur's lemma the center $Z$ of $G$ acts on $V$ by scalar
matrices. On the other hand, $X$ is embedded into the projective space
$\P=\P\big(H^0(X,-K_X)^\vee\big)$, and the action of $G$ on $\P$
induces the initial action of $G$ on $X$. However, $Z$ acts trivially on $\P$,
which is a contradiction.
\end{proof}

\begin{proposition}\label{proposition:g-from-6-to-8}
Let $X$ be a Gorenstein Fano threefold with $\rho(X)=1$, $\iota(X)=1$, and $\g(X)\ge 6$.
Let $G\subset\Aut(X)$ be a $p$-group such that $X$ is a $G$-Fano variety.
Suppose that $G$ does not have fixed points on $X$.
Suppose also that $p\ge 5$.
Then either the group $G$ has a fixed point on $X$, or one of the following cases occurs:
\begin{itemize}
\item[(A)]
$p=5$, $\g(X)=6$, and there is a $G$-invariant anticanonical divisor on $X$;
\item[(B)]
$p=5$, $\g(X)=7$, the threefold $X$ is $\QQ$-factorial,
and there is an irreducible $G$-invariant anticanonical divisor on $X$;
\item[(C)]
$p=5$, $\g(X)=8$, the threefold $X$ is $\QQ$-factorial,
and there are no $G$-invariant anticanonical divisors on $X$;
\item[(D)]
$p=7$, $\g(X)=8$, and there is a $G$-invariant anticanonical divisor on $X$.
\end{itemize}
\end{proposition}
\begin{proof}
By Lemma~\ref{lemma:double-cover-1}
the divisor $-K_X$ is very ample;
alternatively, one can use the results of~\cite{Przhiyalkovskij-Cheltsov-Shramov-2005en}
here. By Lemma~\ref{lemma:g-ge-9} we may assume that~\mbox{$\g(X)\le 8$}.

Assume that $H^0(X,-K_X)$ has no one-dimensional
$G$-invariant subspaces.
By Lemma~\ref{lemma:KX-irreducible-representation} one has $\g(X)+2\ge 2p\ge 10$.
Hence $\g(X)+2=2p=10$, so that $p=5$ and $\g(X)=8$.
This is case~(C) of the lemma. The threefold $X$ is $\QQ$-factorial by Lemma~\ref{lemma-g-equiv1} in this case.

Now assume that $H^0(X,-K_X)$ has a one-dimensional
$G$-invariant subspace. Then
there exists a $G$-invariant divisor $S\in |-K_X|$.

Suppose that $S$ is irreducible.
Then $S$ is as in case~(B) of Lemma~\ref{lemma:invariant-member}.
Let $\Lambda$ be one of the irreducible components of the non-normal locus of the surface $S$.
Then $\Lambda$ is a $G$-invariant non-rational curve.
Let $d=\deg \Lambda$,
let~\mbox{$\Pi=\langle \Lambda\rangle$} be the linear span of $\Lambda$, and let $H$ be
a general hyperplane section of $S$. Then $H$ is an irreducible curve with $p_a(H)=\g(X)$.
Note that the surface $S$ is non-rational, so that the curve $H$ is non-rational as well.
Since~$H$ is singular along $\Lambda\cap H$, we have
$\g(X)\ge d+1$.
Since the curve $\Lambda$ is irreducible and non-rational, we have
$$
\dim \Pi\le d-1\le \g(X)-2.
$$

Consider the case when $d<p$.
By assumption there are no $G$-fixed points on $\Lambda$.
Therefore, $\Lambda$ has no $G$-invariant hyperplane sections, which implies that $\Pi$ has no
$G$-invariant $(\dim\Pi-1)$-dimensional linear subspaces, and thus $\Pi$ also does not
have $G$-fixed points.
By Lemma~\ref{lemma:PGL} we have
\[
\dim \Pi\ge p-1>d-1,
\]
which is a contradiction.

Thus we may assume that $\g(X)-1\ge d\ge p$.
In particular, we have~\mbox{$p\le 7$}.
If $p=7$, then $\g(X)=8$. This is case~(D) of the lemma.

Take general points $P_1,\dots,P_{\g(X)-d}\in S$ and
a general hyperplane section~\mbox{$H'\subset S$} passing through $\Pi$ and these points.
Then
\[
2\g(X)-2=H\cdot H'\le 2d +\g(X)-d=d+\g(X),
\]
which gives $\g(X)\le d+2$.

Therefore, we are left with the case when
$d\ge p=5$ and
$$
d+1\le \g(X)\le d+2.
$$
If $\g(X)=6$, we get
case~(A) of the lemma. If $\g(X)=7$, we get case~(B) of the lemma;
the threefold $X$ is $\QQ$-factorial by Lemma~\ref{lemma-g-equiv1} in this case.
If~\mbox{$\g(X)=8$}, we have $6\le d\le 7$, which contradicts Lemma~\ref{lemma:elliptic}.

Finally, suppose that $S$ is reducible, that is, $S=\sum_{i=1}^N S_i$, $N>1$.
The group $G$ permutes the surfaces
$S_i$ transitively. Hence $N=p^k$ for some $k$.
Thus
\[
\deg S=2\g(X)-2=p^k\deg S_i.
\]
This is possible only if $\deg S_i=2$, $N=p$, and $\g(X)=p+1$.
We see that either $p=5$ and $\g(X)=6$, or $p=7$ and $\g(X)=8$.
The former is case~(A), and the latter is case~(D) of the lemma.
\end{proof}

\begin{remark}
In case~(B) of Proposition~\ref{proposition:g-from-6-to-8}
one has $\rr(G)\le 3$ by Lemma~\ref{lemma:invariant-member}.
In case~(C) of Proposition~\ref{proposition:g-from-6-to-8}
the threefold $X$ has at most~$10$ singular points by~\cite[Theorem 1.3]{Prokhorov-planes}.
This implies that there is a subgroup $G'\subset G$ of index at most $p=5$
such that $G'$ has a fixed point on~$X'$. Resolving this point if it is singular
and using Lemma~\ref{lemma:G-MMP} together with Remark~\ref{remark:Aut-P}
and Lemma~\ref{lemma:GL}, one can show that $\rr(G')\le 3$ and $\rr(G)\le 4$ in this case,
cf. the proof of Theorem~\ref{theorem:main} in~\S\ref{section:proof} below.
\end{remark}

We summarize the most important results of this section as follows.

\begin{corollary}\label{corollary:Gorenstein-Fano-3-folds}
Let $G$ be a $p$-group, and $X$ be a $G$-Fano threefold.
Suppose that $p\ge 11$. Then $G$ has a fixed point on $X$.
\end{corollary}
\begin{proof}
By Lemma~\ref{lemma:rho-large}
we may assume that $\rho(X)=1$, and by Lemma~\ref{lemma:rho-1-iota-large}
we may assume that $\iota(X)=1$.
If $\g(X)\le 5$, the assertion follows from Lemma~\ref{lemma:Fano-3-fold-small-g}.
If $\g(X)\ge 9$, the assertion follows from Lemma~\ref{lemma:g-ge-9}.
Finally, if~\mbox{$6\le \g(X)\le 8$}, the assertion follows from Proposition~\ref{proposition:g-from-6-to-8}.
\end{proof}

\section{Proofs of the main results}
\label{section:proof}

In this section we prove Theorems~\ref{theorem:main} and~\ref{theorem:fixed-point},
and make a couple of concluding remarks.

\begin{proposition}\label{proposition:non-Gorenstein}
Let $X$ be a terminal Fano threefold with at least one non-Gorenstein point.
Let $G\subset\Aut(X)$ be a $p$-group. Suppose that $p\ge 17$.
Then $G$ has a fixed point on $X$.
\end{proposition}
\begin{proof}
Denote by $r_P$ the index of a point $P\in X$, i.e. the minimal positive integer $t$ such that
the divisor $tK_X$ is Cartier near $P$.
Recall that the terminal singularities are isolated, so
there is only a finite number of points with~\mbox{$r_P>1$}.
By \cite{Kawamata-1992bF}, \cite{KMMT-2000} we have
\begin{equation*}
\label{eq-RR-O}
\sum_{P\in X} \left( r_P-\frac 1{r_P}\right)<24.
\end{equation*}
This shows that the number of singular points of $X$ is at most $15$.
These points cannot be permuted by a $p$-group $G$ if $p\ge 17$.
Hence $G$ has a fixed point on $X$.
\end{proof}

\begin{remark}
To improve the result of Proposition~\ref{proposition:non-Gorenstein}
one needs a classification of non-Gorenstein
$\QQ$-Fano threefolds and their automorphism groups.
Note however that the class of $\QQ$-Fano threefolds is huge,
and only some special types of these varieties are classified
(see e.g. \cite{Prokhorov-e-QFano7}, \cite{Prokhorov-Reid}, and references therein).
\end{remark}

Now we are ready to prove our main results.

\begin{proof}[{Proof of Theorem~\xref{theorem:fixed-point}}]
Taking a $G$-equivariant desingularization (see~\cite{Abramovich-Wang})
and keeping in mind that an image of a $G$-fixed point with respect
to any $G$-equivariant morphism is again a $G$-fixed point, we may assume
that the threefold $X$ is smooth.

Applying $G$-equivariant Minimal Model Program to $X$
(which is possible due to an equivariant version of~\cite[Corollary 1.3.3]{BCHM} and
\cite[Theorem~1]{MiyaokaMori}, since rational connectedness
implies uniruledness), we obtain a
$G$-equivariant birational map $f\colon X\dashrightarrow X'$ such that either there
is a $G$-Mori fiber space~\mbox{$\phi\colon X'\to S$}, or~$X'$ is a $G\Q$-Fano threefold.
In the former case $G$ has a fixed point on $X'$
by Corollary~\ref{corollary:MFS-dim-3}. In the latter case
$G$ has a fixed point on $X'$ by Corollary~\ref{corollary:Gorenstein-Fano-3-folds} and
Proposition~\ref{proposition:non-Gorenstein}.
Thus the existence of a $G$-fixed point on $X$
follows from Lemma~\ref{lemma:G-MMP}.
\end{proof}

\begin{proof}[{Proof of Theorem~\xref{theorem:main}}]
Regularizing the action of $G$ and taking an equivariant
desingularization (see e.\,g.~\mbox{\cite[Lemma-Definition~3.1]{Prokhorov-Shramov-2013}}),
we may assume that $X$ is smooth and~\mbox{$G\subset\Aut(X)$}.
By Theorem~\ref{theorem:fixed-point} the group~$G$ has a fixed point on $X$.
Thus there is an embedding~\mbox{$G\hookrightarrow\GL_3(\Bbbk)$} by Remark~\ref{remark:Aut-P}.
Now the assertion follows from Lemma~\ref{lemma:GL}.
\end{proof}

\begin{remark}
If we restrict ourselves to subgroups of $\Cr_3(\Bbbk)$, then it may be possible to strengthen
the bounds. Namely, although rationally connected varieties with relatively large
automorphism groups are sometimes necessarily rational
(as in the case of cubic or quartic hypersurfaces in $\P^4$ with a maximal number of isolated
singularities), in some cases having a large automorphism group
implies non-rationality of a corresponding variety,
cf.~\cite[Remark~7.4.2]{Prokhorov-Shramov-JCr3}, \cite[Theorem~1.2(ii)]{PrzyjalkowskiShramov2016}.
\end{remark}

In any case, we believe that the results of Propositions~\ref{proposition:g-from-6-to-8}
and~\ref{proposition:non-Gorenstein}, and thus of Theorems~\ref{theorem:main} and~\ref{theorem:fixed-point}, can be strengthened.

\bibliographystyle{alpha}
\def\cprime{$'$}

\end{document}